\documentclass[leqno,11pt]{amsart}
\usepackage[utf8]{inputenc}

\usepackage{amsmath}
\usepackage{amsthm}

\usepackage{enumerate}

\usepackage{mathtools}

\usepackage[colorlinks=false]{hyperref}

\usepackage[charter]{mathdesign}
\usepackage{xy}
\xyoption{all}

\usepackage{blkarray} 
\usepackage{multirow}
\makeatletter
\g@addto@macro\bfseries{\boldmath}
\makeatother
\usepackage{color}

\theoremstyle{plain}
    \newtheorem{theorem}[equation]{Theorem}
    \newtheorem{lemma}[equation]{Lemma}
    \newtheorem{corollary}[equation]{Corollary}

    \newtheorem*{theorem*}{Theorem}
    \newtheorem*{proposition*}{Proposition}
    \newtheorem*{corollary*}{Corollary}
    \newtheorem*{lemma*}{Lemma}
    \newtheorem*{conjecture*}{Conjecture}
    \newtheorem{definition-theorem}[equation]{Definition/Theorem}
    \newtheorem{definition-lemma}[equation]{Definition/Lemma}

\theoremstyle{definition}
    
    \newtheorem{example}[equation]{Example}

    \newtheorem*{definition*}{Definition}

    \newcommand{\N}{\mathbb{N}}
    \newcommand{\Z}{\mathbb{Z}}
   \newcommand{\R}{\mathbb{R}}

    \newcommand{\trace}{\operatorname{t}}
    \newcommand{\Trace}{\operatorname{T}}
    \newcommand{\TRACE}{\mathfrak{T}}

\title{If $R^m\cong R^n$, must $m=n$?}
\author{Tyrone Crisp}
\address{Department of Mathematics \& Statistics, University of Maine.
5752 Neville Hall, Room 333.
Orono, ME 04469 USA}
\email{tyrone.crisp@maine.edu}
\date{January 2022}
 
\begin{document}

\begin{abstract}
A fundamental theorem of linear algebra asserts that every basis for the vector space $\R^n$ has $n$ elements. In this expository note we present a theorem of W.~G.~Leavitt describing one way in which this \emph{invariant basis number} property can fail when one does linear algebra over rings, rather than over fields. We give a proof of Leavitt's theorem that combines ideas of  P.~M.~Cohn and A.~L.~S.~Corner into an elementary form requiring only a nodding acquaintance with matrices and modular arithmetic.
\end{abstract}

\maketitle

\subsection*{The invariant basis number property.} We learn and teach in our introductory linear algebra classes that over the real numbers $\R$ only square matrices can be invertible. Indeed, if matrices $X\in M_{m\times n}(\R)$ and $Y\in M_{n\times m}(\R)$ satisfy $XY= I_{m}$ (the $m\times m$ identity matrix) and $YX= I_{n}$, then applying the trace (sum of the diagonal entries) gives
\[
m=\operatorname{trace}( I_{m}) = \operatorname{trace}(XY)=\operatorname{trace}(YX)= \operatorname{trace}( I_{n})=n.
\]
Since a matrix is invertible if and only if its rows form a basis for the appropriate $\R$-vector space we can restate this result by saying that all bases of $\R^n$ have exactly $n$ elements. The same theorem (though not the same proof) applies when $\R$ is replaced by any field: see \cite[\S 8]{Halmos} for an elegant field-independent treatment.

Linear algebra over rings---the theory of modules---is far more complicated than linear algebra over fields, starting with the fact that not every module has a basis. But for modules that do have a basis, do all bases still have the same number of elements? 

To phrase the question in a more precise and elementary form, let $R$ be a ring, not necessarily commutative but always assumed to be associative and to have a multiplicative identity element. For each positive integer $n$ we consider the set $R^n = \{(r_1,\ldots, r_n)\ |\ r_i\in R\}$, equipped with operations of component-wise addition and component-wise multiplication on the left by elements of $R$. Just as in the case of $R=\R$ we say that a mapping $F: R^m\to R^n$ is \emph{linear} if it satisfies
\[
\begin{aligned}
 F( r_1+s_1,\ldots, r_m+s_m) &= F(r_1,\ldots,r_m) + F(s_1,\ldots, s_m)\\ \text{and} \qquad F(rr_1,\ldots, rr_m) &= rF(r_1,\ldots, r_m)
\end{aligned}
\]
for all $r, r_i,s_i\in R$. The standard proof from Matrices 101, with a little extra care taken to account for the fact that our ring $R$ might be noncommutative, shows that every linear map $R^m\to R^n$ has the form $(r_1,\ldots, r_m)\mapsto (r_1,\ldots,r_m)X$ for some $m\times n$ matrix $X$ with entries in $R$. We write $R^m\cong R^n$ if there exists a bijective linear map from $R^m$ to $R^n$, or equivalently, if there are matrices $X\in M_{m\times n}(R)$ and $Y\in M_{n\times m}(R)$ satisfying $XY= I_{m}$ and $YX= I_{n}$. Our question now becomes, if $R^m\cong R^n$ must $m=n$?

We noted above that if $R$ is a field then, yes, $R^m\cong R^n$ implies $m=n$. Using this and  the fact that every commutative ring has a quotient that is a field (Krull's theorem), one can also prove that if $R$ is any commutative ring then $R^m\cong R^n$ implies $m=n$. In fact, most of the rings that one meets in a first course of algebra have this property of \emph{invariant basis number}: for instance,  all matrix rings $M_{d\times d}(R)$ and group rings $R[G]$ over a commutative ring $R$ have invariant basis number, as do all subrings of rings with this property. See \cite[Section 2]{Leavitt-module-type} for proofs of these assertions and more. The following example of a ring that does not have invariant basis number was pointed out by Dieudonn\'e \cite{Dieudonne} and Everett \cite{Everett}:

\begin{example}\label{example:11}
Let $\N$ be the set of positive integers. Consider the set $R$ of $\N\times \N$ matrices $A=(a_{i,j})_{i,j=1}^\infty$, with integer entries $a_{i,j}\in \Z$, having only a finite number of nonzero entries in each row and in each column. We add matrices entry-wise, and we multiply them by using the formula $(AB)_{i,j} = \sum_{k=1}^\infty a_{i,k}b_{k,j}$. Since for each $i$ the $i$th row of $A$ contains only finitely many nonzero entries $a_{i,k}$,  this infinite sum has only finitely many nonzero terms, so it is well-defined as an actual sum rather than as a limit. (This will be the case with all of the infinite sums that appear in this paper.) Moreover,  for each of the finitely many $k$ with $a_{i,k}\neq 0$ there are only finitely many $j$ with $b_{k,j}\neq 0$, and so the $i$th row of $AB$ contains only finitely many nonzero entries. A similar argument shows that each column of $AB$ contains only finitely many nonzero entries,  so $AB$ lies in $R$. Thus $R$ is a ring, with the $\N\times \N$ identity matrix $I_{\infty}$ as multiplicative identity.

For each matrix $A\in R$ we let $A_{\mathrm{odd}}\in R$ be the matrix whose columns are the odd-numbered columns of $A$, and we let $A_{\mathrm{even}}\in R$ be the matrix whose columns are the even-numbered columns of $A$. The map $R\to R^2$ given by $A\mapsto (A_{\mathrm{odd}}, A_{\mathrm{even}})$ is linear, because the $j$th column of a matrix product $BA$ is the product of $B$ with the $j$th column of $A$. Since each column of $A$ appears in precisely one of the matrices $A_{\mathrm{odd}}$ and $A_{\mathrm{even}}$ this map is  a bijection, and so we have $R\cong R^2$. Now the map $(A,B)\mapsto(A_{\mathrm{odd}},A_{\mathrm{even}},B)$ is an isomorphism $R^2\xrightarrow{\cong} R^3$, and proceeding by induction we see  that $R^m\cong R^n$ for all positive integers $m,n$.
\end{example}

This example represents an extreme failure of the invariant basis number property, and it is natural to wonder whether a ring can fail to have this property in a less extreme way: is there a ring $R$ having $R^m\cong R^n$ for some positive integers $m\neq n$, but not for all  $m$ and $n$? The following special case of {\cite[Theorem 8]{Leavitt-module-type}} gives an affirmative answer to this question.

\begin{theorem}[Leavitt]\label{thm:Leavitt}
For each positive integer $k$ there is a ring $R$ with the property that  $R^m\cong R^n$ if and only if $m\equiv n\bmod k$.
\end{theorem}

This theorem is just one of Leavitt's discoveries about the invariant basis number property and the ways in which this property can fail. Leavitt also proved in \cite{Leavitt-modules-without}, for instance, that for each positive integer $k$ there is a ring $R$ for which $R^m\cong R^n$ if and only if $\min(m,n)\geq k$. This last result turns out to have quite a different flavour to Theorem \ref{thm:Leavitt},  which will be our focus. See \cite{Berrick-Keating} for an accessible introduction to some aspects of the invariant basis number property that are not discussed here.

Leavitt proved Theorem \ref{thm:Leavitt} by an intricate calculation inside a certain universal ring; we shall briefly discuss Leavitt's rings below. Cohn \cite{Cohn} and Corner \cite{Corner-trace} later found shorter proofs of Theorem \ref{thm:Leavitt} using trace functions on Leavitt's ring. In \cite{Corner} Corner gave yet another proof of Leavitt's theorem, that replaces Leavitt's universal ring by a concrete ring of matrices, and that eschews the use of traces in order to make the proof applicable to algebras over arbitrary fields. 

We shall present a short proof of Theorem \ref{thm:Leavitt} that applies the trace techniques of \cite{Cohn} and \cite{Corner-trace} to the matrix ring of \cite{Corner}. We make no claim to novelty for the ideas behind the proof, but we hope that the ruthlessly elementary path taken here  will have pedagogical value in making Leavitt's striking result, and the essential simplicity of Cohn's and Corner's arguments, accessible to mathematicians in the very early stages of their algebraic careers. 

\subsection*{A proof of Leavitt's theorem.}  Fix a positive integer $k$, and consider as in Example \ref{example:11} the ring  $R$ of $\N\times \N$ matrices  with only finitely many nonzero entries in each row and in each column. We partition the matrices in $R$ into blocks, as follows:
\begin{equation*}\label{eq:partition}
{\small
 \left[\begin{array}{@{}c|ccccc|ccccc|ccccc|ccc @{}}
1\times 1 & \multicolumn{5}{c|}{1\times k} & \multicolumn{5}{c|}{1\times k} & \multicolumn{5}{c|}{1\times k}  &   \multicolumn{3}{c}{\cdots} \\ \hline 
&&&&&&&&&&&&&&&&&& \\
 &&&&&&&&&&&&&&&&&& \\
{\smash{\raisebox{\normalbaselineskip}{$k\times 1$}}} &  \multicolumn{5}{c|}{\smash{\raisebox{\normalbaselineskip}{$k\times k$}}} & \multicolumn{5}{c|}{\smash{\raisebox{\normalbaselineskip}{$k\times k$}}} & \multicolumn{5}{c|}{\smash{\raisebox{\normalbaselineskip}{$k\times k$}}} & \multicolumn{3}{c}{\smash{\raisebox{\normalbaselineskip}{$\cdots$}}} \\ \hline
&&&&&&&&&&&&&&&&&& \\
 &&&&&&&&&&&&&&&&&& \\
{\smash{\raisebox{\normalbaselineskip}{$k\times 1$}}} &  \multicolumn{5}{c|}{\smash{\raisebox{\normalbaselineskip}{$k\times k$}}} & \multicolumn{5}{c|}{\smash{\raisebox{\normalbaselineskip}{$k\times k$}}} & \multicolumn{5}{c|}{\smash{\raisebox{\normalbaselineskip}{$k\times k$}}} & \multicolumn{3}{c}{\smash{\raisebox{\normalbaselineskip}{$\cdots$}}} \\ \hline
&&&&&&&&&&&&&&&&&& \\
{\smash{\raisebox{.5\normalbaselineskip}{$\vdots$}}} &  \multicolumn{5}{c|}{\smash{\raisebox{.5\normalbaselineskip}{$\vdots$}}} & \multicolumn{5}{c|}{\smash{\raisebox{.5\normalbaselineskip}{$\vdots$}}} & \multicolumn{5}{c|}{\smash{\raisebox{.5\normalbaselineskip}{$\vdots$}}} & \multicolumn{3}{c}{\smash{\raisebox{.5\normalbaselineskip}{$\cdots$}}}  \\ 

\end{array}\right]
}
\end{equation*}
We write $A^{m,n}$ to denote the $m,n$ block of the matrix $A$, where as usual $m$ is the vertical coordinate and $n$ is the horizontal coordinate. So for example, for $k=1$ each $A^{m,n}$ is the $1\times 1$ matrix $\begin{bmatrix} a_{m,n} \end{bmatrix}$, while for $k=2$ and for the matrix
\begin{equation*}\label{eq:r-example}
A =  \left[\begin{array}{@{}c|cc|cc|cc|cc|cc|cc @{}}
1 & 2 & 3 & 4 & 5 & 6 & 7 & 8 & 9 & 10 & 11 & 12 & \cdots \\ \hline 
2 & 4 & 6 & 8 & 10 & 12 & 14 & 16 & 18 & 20 & 22 & 24 &  \cdots \\
3 & 6 & 9 & 12 & 15 & 18 & 21 & 24 & 27 & 30 & 33 & 36 &  \cdots \\ \hline
4 & 8 & 12 & 16 & 20 & 24 & 28 & 32 & 36 & 40 & 44 & 48 &  \cdots \\
5 & 10 & 15 & 20 & 25 & 30 & 35 & 40 & 45 & 50 & 55 & 60 & \cdots \\ \hline
\vdots & \vdots & \vdots & \vdots & \vdots & \vdots &  \vdots & \vdots &  \vdots & \vdots &  \vdots & \vdots & \ddots
\end{array}\right]
\end{equation*}
(where all entries not pictured are zero) we have
\[
A^{1,1} = \begin{bmatrix} 1 \end{bmatrix},\quad A^{1,3}= \begin{bmatrix} 4 & 5 \end{bmatrix},\quad A^{2,1}=\begin{bmatrix} 2\\ 3\end{bmatrix},\quad A^{3,4}= \begin{bmatrix} 24 & 28 \\ 30 & 35 \end{bmatrix},
\] 
and so on.

\begin{definition*} Let $R_k$  be the subset of $R$ consisting of those matrices $A$ in which all but finitely many of the $k\times k$ blocks $A^{m,n}$ (for $m,n\geq 2$) are scalar matrices: that is, they have the form $c  I_{k}$ for some integer $c$, depending on $m$ and $n$.
\end{definition*} 

It is a straightforward matter to show that the blocks of a product of two matrices are given by the formula $(AB)^{m,n} = \sum_{p=1}^\infty A^{m,p}B^{p,n}$ where, as always, only finitely many of the summands are nonzero. Since sums and products of scalar matrices are again scalar matrices, $R_k$ is a subring of $R$. All of the $k\times k$ blocks of the identity matrix  are scalars, so $R_k$ contains the identity. When $k=1$ the restriction on the $k\times k$ blocks is no restriction at all, and so $R_1=R$.

\begin{lemma}\label{lem:1k}
 If $m$ and $n$ are positive integers with $m\equiv n\bmod k$, then $R_k^m \cong R_k^n$.
\end{lemma}

\begin{proof}
We first use a generalisation of the odd/even construction from Example \ref{example:11} to prove that $R_k\cong R_k^{k+1}$. For each $A\in R_k$ and each $l=1,\ldots,k+1$ we let $A_l\in R$ be the matrix whose first column is the $l$th column of $A$, and whose blocks $A_l^{m,n}$ for $n\geq 2$ are given by the formula 
\[
A_l^{m,n}=A^{m,n+l+k(n-2)}.
\] 

For instance, when $k=1$ we have $A_1=A_{\mathrm{odd}}$ and $A_2 = A_{\mathrm{even}}$, while when $k=2$ the matrix $A_1$ consists of the first column of $A$, followed by the column of width-$2$ blocks $A^{m,3}$, followed by the column of blocks $A^{m,6}$, then $A^{m,9}$, etc; the matrix $A_2$ consists of the second column of $A$ followed by the column of width-$2$ blocks $A^{m,4}$, then the column of blocks $A^{m,7}$, etc; and the matrix $A_3$ consists of the third column of $A$ followed by the column of blocks $A^{m,5}$, then $A^{m,8}$, and so on. 

Since each $k\times k$ block of $A_l$ is one of the blocks of $A$, all but finitely many of which are scalars, we have $A_l\in R_k$ for each $l$.
 Now the map $R_k\mapsto R_k^{k+1}$ given by $A\mapsto (A_1,\ldots,A_{k+1})$ is an isomorphism, for the same reason as in the case of $k=1$ considered in Example \ref{example:11}. Repeatedly applying this isomorphism to the first coordinate, again as in Example \ref{example:11}, we get an isomorphism $R_k^m \cong R_k^{m+hk}$ for each pair of positive integers $m$ and $h$, and so $R_k^m\cong R_k^n$ whenever $m\equiv n\bmod k$.
\end{proof}

Before proceeding, it is an instructive exercise to attempt to adapt the above proof to show that $R_k\cong R_k^n$ for an arbitrary positive integer $n$. One should not spend too long on the attempt, however:

\begin{theorem}\label{thm:1k} 
$R_k^m \cong R_k^n$ if and only if $m\equiv n\bmod k$. Thus the ring $R_k$ witnesses the truth of Theorem \ref{thm:Leavitt}.
\end{theorem}

One implication in Theorem \ref{thm:1k} is provided by Lemma \ref{lem:1k}. We shall prove the other implication using a modification of the argument with the trace that we used to prove that $\R^m\cong \R^n$ implies $m= n$. Some alterations are necessary because, for one thing, we now wish to use the trace to detect not whether two integers are equal, but whether they are congruent modulo $k$;  our trace will accordingly take values in the ring $\Z/k\Z$ of residue classes modulo $k$. Secondly, since we want to apply our trace to matrices with entries in the ring $R_k$---that is, matrices whose entries are themselves infinite matrices, partitioned into finite blocks---our definition will involve summing the diagonal entries at three different scales. Finally, since one of these scales involves infinite matrices we will need to be careful in how we perform the summation, and it is at this point that we shall use the condition that the elements of $R_k$ have only finitely many non-scalar blocks.

\begin{definition*}
{\bf(1)} For each finite square matrix $M$ with entries in $\Z$ we define $\trace_k(M)\in \Z/k\Z$ to be the residue class of the sum of the diagonal entries of $M$. Notice that for a $k\times k$ scalar matrix $c I_{k}$ we have $\trace_k(c I_{k}) = ck = 0$ in $\Z/k\Z$. 

\medskip

\noindent{\bf(2)} Now for each infinite matrix $A\in R_k$ we define $\Trace_k(A)\in \Z/k\Z$ to be
\[
\Trace_k(A) = \sum_{n=1}^\infty \trace_k(A^{n,n}),
\]
the sum of the residue classes  of the traces of the diagonal blocks of $A$. Since all but finitely many of these blocks are $k\times k$ scalar matrices, which have trace zero modulo $k$, the sum defining $\Trace_k(A)$ has only finitely many nonzero terms. For example, the identity matrix $I_\infty\in R_k$ has $\Trace_k(I_\infty)=1+0+0+0+\cdots = 1\in \Z/k\Z$.

\medskip

\noindent{\bf(3)} Finally, for each finite square matrix $W\in M_{d\times d}(R_k)$ with entries in $R_k$ we define $\TRACE_k(W)\in \Z/k\Z$ by
\[
\TRACE_k(W) = \sum_{l=1}^d \Trace_k\left(W_{l,l}\right),
\]
the result of applying the function $\Trace_k :R_k\to \Z/k\Z$ to the sum of the diagonal entries of $W$. For example, the $m\times m$ identity matrix $ I_{m}\in M_{m\times m}(R_k)$ has $\TRACE( I_{m})=m\Trace_k(I_\infty)=m$ in $\Z/k\Z$.
\end{definition*}

It is easy to see that our trace map $\TRACE_k$ satisfies $\TRACE_k(W+Z)=\TRACE_k(W)+\TRACE_k(Z)$. 
Examining the compatibility of $\TRACE_k$ with multiplication, we first note that a familiar property of the matrix trace gives  $\trace_k(MN)=\trace_k(NM)$ for all $M\in M_{m,n}(\Z)$ and $N\in M_{n,m}(\Z)$, and this implies that $\Trace_k(AB)=\Trace_k(BA)$ for all $A,B\in R_k$ because
\begin{multline*}
 \Trace_k(AB)=\sum_{n=1}^\infty \trace_k  (AB)^{n,n}  = \sum_{n=1}^\infty \sum_{m=1}^\infty \trace_k( A^{n,m}B^{m,n}) \\ = \sum_{m=1}^\infty \sum_{n=1}^\infty \trace_k(B^{m,n}A^{n,m}) =\sum_{m=1}^\infty \trace_k(BA)^{m,m}=\Trace_k(BA).
\end{multline*}
(Note that both of the apparently infinite sums are actually finite, so there is no problem with interchanging the order of summation.) A similar computation shows that the identity $\Trace_k(AB)=\Trace_k(BA)$ implies in turn that $\TRACE_k(XY)=\TRACE_k(YX)$ for all $X\in M_{m\times n}(R_k)$ and $Y\in M_{n\times m}(R_k)$.

Having constructed a trace function with the right properties, we now prove Theorem \ref{thm:1k} by an argument that is essentially identical to our proof that $\R^m\cong \R^n$ implies $m=n$:

\begin{proof}[Proof of Theorem \ref{thm:1k}]
Suppose that $R_k^m \cong R_k^n$ for positive integers $m,n$, so that we have matrices $X\in M_{m\times n}(R_k)$ and $Y\in M_{n\times m}(R_k)$ satisfying $XY= I_{m}$ and $YX= I_{n}$. Applying $\TRACE_k$ gives a string of equalities in $\Z/k\Z$,
\[
m=\TRACE_k( I_{m}) = \TRACE_k(XY)=\TRACE_k(YX)=\TRACE_k( I_{n})=n,
\]
showing that $m\equiv n\bmod k$.
\end{proof}

\subsection*{Leavitt's rings.} In \cite[Section 3]{Leavitt-module-type} Leavitt defined and studied, for each positive integer $p$,  a ring $L_p$ (not Leavitt's notation) having the property that a ring $R$ admits a linear isomorphism  $R\cong R^p$ if and only if it admits a $1$-preserving homomorphism of rings $L_p\to R$. 

To construct such a ring (using slightly more algebraic background than we have assumed up to this point), start with the ring $\Z\langle x_1,\ldots, x_p, y_1,\ldots, y_p\rangle$ of polynomials in noncommuting variables $x_1,\ldots,x_p,y_1,\ldots, y_p$, and form the quotient by the two-sided ideal generated by the polynomials $\sum_{i=1}^{p} x_i y_i - 1$ and $y_i x_j - \delta_{i,j}$ (where $\delta_{i,j}$ is $1$ if $i=j$ and $0$ otherwise). Call this quotient $L_p$. By the universal mapping property of quotient rings, a ring $R$ admits a $1$-preserving homomorphism $L_p \to R$ if and only if $R$ contains elements $x_1',\ldots,x_{p}',y_1',\ldots, y_{p}'\in R$ satisfying the relations $\sum_{i=1}^{p} x'_i y'_i=1$ and $y'_i x'_j = \delta_{i,j}$. These last relations are equivalent to the relations $XY= I_{1}$ and $YX= I_{p}$ on the matrices
\[
X = \begin{bmatrix} x'_1 & \cdots & x'_{p}\end{bmatrix}\quad \text{and}\quad Y = \begin{bmatrix} y'_1 \\ \vdots \\ y'_{p}\end{bmatrix},
\]
and we observed above that the existence of matrices over $R$ satisfying these relations is equivalent to the existence of an isomorphism $R\cong R^{p}$. 

Since $L_p$ itself admits the identity morphism $L_p\to L_p$, we have  $L_p\cong L_p^p$. 
It is not obvious from the construction that $L_p\not\cong L_p^h$ for $1<h<p$, but this follows from Theorem   \ref{thm:1k} by an argument of Leavitt's (cf.~\cite[Theorem 2]{Leavitt-module-type}):

\begin{corollary}
Let $p$ be a positive integer, and let $L_p$ be the  ring defined above. Then $L_p^m\cong L_p^n$ if and only if $m\equiv n \bmod (p-1)$.
\end{corollary}

\begin{proof}
When $p=1$ the relations $x_1y_1=y_1x_1=1$ ensure that the ring $L_1$ is commutative (it is the ring of  Laurent polynomials in one variable with integer coefficients), and so $L_1$ has  invariant basis number: $L_1^m\cong L_1^n$ if and only if $m=n$, or in other words, if and only if $m\equiv n \bmod 0$. 

Now suppose $p\geq 2$.  We noted above that $L_p\cong L_p^{p}$, and  an inductive argument as in Lemma \ref{lem:1k} shows that  $L_p^m \cong L_p^n$ whenever $m\equiv n \bmod(p-1)$. Conversely, suppose that $L_p^m\cong L_p^n$ for some positive integers $m$ and $n$, so that we have matrices $X\in M_{m,n}(L_p)$ and $Y\in M_{n,m}(L_p)$ with $XY= I_{m}$ and $YX= I_{n}$. The fact that $R_{p-1}\cong R_{p-1}^{p}$ (Lemma \ref{lem:1k}) ensures that we have a $1$-preserving homomorphism of rings $f:L_p\to R_{p-1}$, and applying this homomorphism to each entry of the matrices $X$ and $Y$ gives matrices $f(X)\in M_{m,n}(R_{p-1})$ and $f(Y)\in M_{n,m}(R_{p-1})$ satisfying $f(X)f(Y)= I_{m}$ and $f(Y)f(X)= I_{n}$. These matrices yield a linear isomorphism $R_{p-1}^m\cong R_{p-1}^n$, and so Theorem \ref{thm:1k}  implies that $m\equiv n \bmod(p-1)$.
\end{proof}

\subsection*{Outlook.} The past two decades have seen an explosion of interest in Leavitt's rings and subsequent elaborations: at the time of writing, Google Scholar lists 30 citations of \cite{Leavitt-module-type} from the twentieth century, and 225 from the twenty-first. The reader who is curious about these developments might begin with \cite{Abrams-whatis} and proceed to \cite{Abrams-decade}.

\bibliographystyle{plain}
\bibliography{IBN.bib}

 \end{document}